\documentclass[11pt]{article}

\usepackage[T1]{fontenc}
\usepackage[utf8]{inputenc}
\usepackage[margin=27mm]{geometry}
\usepackage{amsmath,amssymb,amsthm}
\usepackage{longtable,booktabs,array,calc}
\usepackage{microtype}
\usepackage[hidelinks]{hyperref}
\IfFileExists{xurl.sty}{\usepackage{xurl}}{}
\providecommand{\tightlist}{%
  \setlength{\itemsep}{0pt}\setlength{\parskip}{0pt}}

\newtheorem{theorem}{Theorem}[section]
\newtheorem{proposition}[theorem]{Proposition}
\newtheorem{corollary}[theorem]{Corollary}

\newcommand{\citeproctext}{}
\newcommand{\citeproc}[2]{#2}
\newlength{\cslhangindent}
\newlength{\csllabelwidth}
\newenvironment{CSLReferences}[2]
 {\begin{list}{}{%
  \setlength{\itemindent}{0pt}
  \setlength{\leftmargin}{0pt}
  \setlength{\parsep}{0pt}
  \ifodd #1
   \setlength{\leftmargin}{\cslhangindent}
   \setlength{\itemindent}{-\cslhangindent}
  \fi
  \setlength{\itemsep}{#2\baselineskip}}}
 {\end{list}}

\makeatletter
\def\@biblabel#1{}
\makeatother

\hypersetup{
  pdftitle={Certified Countermodels in Profile and Incidence Fibres of Diamond-Induced Edge Partitions},
  pdfauthor={Ivan Khalamendyk}
}

\title{Certified Countermodels in Profile and Incidence Fibres of\\
Diamond-Induced Edge Partitions}
\author{Ivan Khalamendyk\\
\small Independent researcher, Ukraine}
\date{21 August 2026}

\begin{document}
\maketitle

\begin{abstract}
Let \(G=(V,E)\) be a finite loopless directed graph. Each directed
two-step diamond identifies its two pairs of opposite edges, and the
connected components of the resulting auxiliary graph on \(E\) define a
canonical partition \(\Pi_{\mathrm{opp}}(G)\). We ask whether a finite
relational certificate built from the blocks of an edge partition
identifies this canonical partition.

We prove the functoriality and a universal coarsening property of
\(\Pi_{\mathrm{opp}}(G)\), count its complete fixed-profile fibre and
its exact radius-one transposition neighbourhood, and reduce the parity
condition in the certificate to an odd closed walk in a directed
\(\mathbb Z_2\)-gain graph. In a fixed catalogue, four labelled
target-positive rows, representing three isomorphism types, each admits
a noncanonical positive partition obtained by a single cross-block
transposition. In one row a stronger positive comparison partition also
preserves every per-vertex, fixed-role incoming and outgoing count; it
arises from a 12-edge alternating trade. The corresponding incidence
fibre is a singleton in the other three rows.

A deterministic 64-index family of comparison partitions yields 78
bounded positives among 256 nonidentity partitions. On the same profile
fibre, one member has a complete 885-element relation semigroup and no
write-preserve-use certificate, giving an exact negative. Thus the
certificate has genuine but intermediate selectivity: it is neither
determined by block sizes nor specific to the canonical target. The
claims are finite statements in computational combinatorics and are
supported by explicit, independently checkable certificates.

\end{abstract}

\section{Introduction}\label{introduction}

A structured object is often tested by a certificate assembled from
several derived features. A positive certificate may show that the
features coexist, yet it need not identify the construction that
originally motivated them. The distinction is elementary but important:
existence and specificity are different mathematical questions.

This paper studies that distinction for partitions of the directed edges
of a finite graph. The target partition is not arbitrary. It is
generated by a local diamond rule: opposite sides of every directed
two-step diamond pattern are forced into the same block, and the rule is
closed transitively. Each block is then read as a binary relation on the
vertex set. Products of those relations can write, preserve and later
distinguish finite state pairs; the resulting record systems admit
translations, and translations carry a parity. The final positive
certificate contains a closed directed walk of odd total parity.

The WPU-parity certificate is used here as a concrete stress test for
identifiability. It combines finite relation-semigroup dynamics, a
stable behavioural quotient, and \(\mathbb Z_2\)-gain parity; the latter
two connect with standard refinement and gain-graph constructions
(\citeproc{ref-PaigeTarjan1987}{Paige and Tarjan 1987};
\citeproc{ref-Zaslavsky1982}{Zaslavsky 1982}). The catalogue and
certificate arose in an exploratory study of finite term-rewriting
systems with persistent distinctions and transport between record
systems. That computation produced four positive targets. Its
exploratory origin is not used as evidence here: the catalogue, target
partitions and certificate predicate were fixed before profile-matched
comparison partitions were generated. In this paper, a \emph{control} is
an alternative partition of the same edge set constructed subject to the
stated invariant, either the block-size profile or the fixed-role
incidence table; the term has no probabilistic meaning. The scientific
question addressed here is deliberately narrower:

\begin{quote}
Does the relational certificate identify the canonical diamond-induced
edge partition among partitions with the same edge carrier and the same
multiset of block sizes?
\end{quote}

The answer is no for every retained row, and the failure is local. A
single transposition of two edges between two blocks is enough. This is
the smallest nonidentity move in the fixed-profile partition graph.

\subsection{Contributions}\label{contributions}

The paper has six contributions.

\begin{enumerate}
\def\labelenumi{\arabic{enumi}.}
\tightlist
\item
  We give a self-contained graph-theoretic definition of the
  opposite-diamond partition and prove its functoriality and uniqueness
  among partitions with the same block count that preserve all
  generating diamond constraints.
\item
  We describe the complete profile fibre, prove its exact cardinality,
  and count the radius-one neighbourhood in the unlabelled transposition
  graph.
\item
  We define a stronger fixed-role incidence fibre and prove the
  alternating two-role trade lemma. One catalogue row has a 12-edge
  incidence-matched positive countermodel, while three rows are exactly
  incidence-rigid.
\item
  We formulate the relational property as a finite certificate on binary
  relations. Its final parity condition is a directed
  \(\mathbb Z_2\)-gain-cycle condition.
\item
  We provide four exact radius-one countermodels, one for each labelled
  positive row. They represent three target isomorphism types. The four
  witnesses are checked by the primary implementation and by an
  independent standard-library implementation.
\item
  We report a complete evaluation of a previously declared 64-index
  control family and an exact negative with a complete 885-element
  relation semigroup. This separates non-specificity from trivial
  profile dependence.
\end{enumerate}

\subsection{Scope}\label{scope}

All theorem statements refer either to arbitrary finite directed graphs
or to the explicit finite catalogue supplied with this paper. A positive
result is based on a finite witness and remains valid even when the
relation semigroup search did not reach a fixed point. A missing witness
is called negative only when the relevant search is proved complete.
Fractions in the deterministic 64-index family are not sampling
probabilities for the full profile fibre.

The incidence fibre uses the target role labels as fixed labels;
quotienting by equal-size role renamings is discussed separately.

\section{Diamond-induced partitions and profile fibres}\label{diamond-induced-partitions-and-profile-fibres}

\subsection{Directed diamond patterns}\label{directed-diamond-patterns}

Throughout, \(G=(V,E)\) is a finite loopless directed graph without
parallel edges. A \textbf{directed two-step diamond pattern} is a
quadruple of distinct vertices \((x,a,b,y)\) such that

\[
 (x,a),(x,b),(a,y),(b,y)\in E.
\]

The diamond contributes two unordered pairs of opposite edges,

\[
 \{(x,a),(b,y)\},\qquad \{(x,b),(a,y)\}.
\]

Define the undirected auxiliary graph \(\Omega(G)\) by taking \(E\) as
its vertex set and joining two directed edges whenever they form one of
these opposite pairs. The \textbf{opposite-diamond partition} is

\[
 \Pi_{\mathrm{opp}}(G)=\pi_0(\Omega(G)),
\]

the set of connected components of \(\Omega(G)\).

This is not an induced-pattern condition: additional arcs among the four
vertices are allowed. Equivalence relations generated by opposite sides
of squares are standard in event-structure, median-graph and
graph-product settings (\citeproc{ref-Chepoi2012}{Chepoi 2012};
\citeproc{ref-HellmuthEtAl2015}{Hellmuth et al. 2015}). Here we use only
the directed two-step pattern above.

\subsection{Edge partitions and their profiles}\label{edge-partitions-and-their-profiles}

An edge partition is an unlabelled collection

\[
 \Pi=\{C_1,\ldots,C_q\}
\]

of nonempty, pairwise disjoint subsets whose union is \(E\). Its profile
is the multiset

\[
 \lambda(\Pi)=\{\!\{|C_1|,\ldots,|C_q|\}\!\}.
\]

For a fixed profile \(\lambda=(s_1,\ldots,s_q)\), the profile fibre is

\[
 \mathcal P_\lambda(E)
 =\{\Pi:\lambda(\Pi)=\lambda\}.
\]

If \(m_r\) is the multiplicity of the part size \(r\), then

\[
 |\mathcal P_\lambda(E)|
 =\frac{|E|!}{\prod_{i=1}^{q}s_i!\prod_{r\ge 1}m_r!}.
\tag{2.1}
\]

The factors \(s_i!\) remove order inside each block and \(m_r!\) removes
the order of equal-sized blocks.

\subsection{Co-membership distance}\label{co-membership-distance}

For an edge partition \(\Pi\), let

\[
 M(\Pi)=\bigl\{\{e,f\}:e\ne f\text{ and }e,f\text{ belong to one block of }\Pi\bigr\}.
\]

We use the exact co-membership distance

\[
 d_{\mathrm{pair}}(\Pi,\Pi')
 =|M(\Pi)\triangle M(\Pi')|.
\tag{2.2}
\]

If \(\lambda(\Pi)=\lambda(\Pi')\), then

\[
 |M(\Pi)\setminus M(\Pi')|
 =|M(\Pi')\setminus M(\Pi)|,
\tag{2.3}
\]

because \(|M(\Pi)|=\sum_i\binom{s_i}{2}\) depends only on the profile.

\subsection{The profile-preserving transposition graph}\label{the-profile-preserving-transposition-graph}

Let \(\Pi\in\mathcal P_\lambda(E)\). Choose distinct blocks
\(A,B\in\Pi\), an edge \(e\in A\), and an edge \(f\in B\). Replacing

\[
 A,B\quad\text{by}\quad A-e+f,\ B-f+e
\tag{2.4}
\]

preserves the profile. After discarding moves that reproduce the same
unlabelled partition, these moves define the radius-one neighbours of
\(\Pi\) in the transposition graph \(\mathcal S_\lambda(E)\). Write
\(\Pi\sim_{\mathcal S_\lambda(E)}\Pi'\) when the corresponding vertices
are adjacent in this graph.

Let \(m_1\) and \(m_2\) be the numbers of singleton and two-element
blocks. Every vertex of \(\mathcal S_\lambda(E)\) has degree

\[
 \deg\mathcal S_\lambda(E)
 =\sum_{i<j}s_is_j-\binom{m_1}{2}-2\binom{m_2}{2}.
\tag{2.5}
\]

Indeed, \(\sum_{i<j}s_is_j\) counts raw cross-block edge pairs. Swapping
two singletons returns the same unlabelled partition. For two
two-element blocks, the four raw swaps occur in complementary pairs and
give only two neighbours. For every remaining nontrivial move, the two
blocks of \(\Pi\) absent from \(\Pi'\) identify \(\{A,B\}\). If their
sizes differ, the replacement blocks are distinguished by size; if their
common size is at least three, each replacement block contains at least
two elements of a unique original block. In either case the set
differences recover \(e\) and \(f\), so neither a different original
block pair nor a different swap gives the same \(\Pi'\). This proves
(2.5).

If the swapped blocks have sizes \(a\) and \(b\), then

\[
 d_{\mathrm{pair}}(\Pi,\Pi')=2(a+b-2).
\tag{2.6}
\]

\subsection{Fixed-role incidence fibres}\label{fixed-role-incidence-fibres}

For this subsection, fix labels \(1,\ldots,q\) on the blocks and
identify a partition with a role assignment \(c:E\to[q]\). For a vertex
\(v\) and role \(r\), write

\[
 d^+_r(v;c)=|\{(v,w)\in E:c(v,w)=r\}|,\qquad
d^-_r(v;c)=|\{(u,v)\in E:c(u,v)=r\}|.
\tag{2.7}
\]

For an assignment whose role fibres are nonempty, write

\[
 \Pi(c)=\{c^{-1}(r):r\in[q]\}
\]

for its underlying unlabelled edge partition.

The \textbf{fixed-role incidence fibre} of \(c\) is

\[
 \mathcal I(c)=
 \{c':d^\pm_r(v;c')=d^\pm_r(v;c)
 \text{ for every }v\in V,\ r\in[q]\}.
\tag{2.8}
\]

Thus \(\mathcal I(c)\) preserves the complete per-vertex, per-role
incoming and outgoing margin table, not merely the role-size multiset.

To describe elementary moves in this fibre, replace every vertex \(v\)
by a tail copy \(v^+\) and a head copy \(v^-\). Every directed edge
\((u,v)\) becomes the bipartite edge \(u^+v^-\). Fix two roles \(r,s\).

\begin{proposition}[Alternating two-role trades]\label{proposition-2.1-alternating-two-role-trades}

Reversing the roles \(r\) and \(s\) along an even cycle whose edge roles
alternate \(r,s,r,s,\ldots\) produces another assignment in
\(\mathcal I(c)\). Conversely, if \(c'\in\mathcal I(c)\) and every
changed edge changes from \(r\) to \(s\) or from \(s\) to \(r\), then
the changed bipartite edges decompose into edge-disjoint alternating
closed trails, and hence into alternating cycles.

\end{proposition}
\begin{proof} At each tail or head copy on an alternating cycle, one
incident \(r\)-edge becomes an \(s\)-edge and one incident \(s\)-edge
becomes an \(r\)-edge, so every margin is unchanged. Conversely, at
every bipartite vertex the number of changed \(r\to s\) edges equals the
number of changed \(s\to r\) edges. Pairing opposite change types
locally decomposes the finite balanced change graph into alternating
closed trails; repeated vertices can be split to obtain cycles.
\end{proof}

This is the familiar switch principle behind fixed-margin fibres and
contingency-table moves (\citeproc{ref-DiaconisSturmfels1998}{Diaconis
and Sturmfels 1998}). Alternating colour flips also underlie Markov
chains on proper bipartite edge-colouring solution spaces
(\citeproc{ref-HongMiklos2023}{Hong and Miklós 2023}). Our role
assignments need not be proper colourings; we use the alternating-cycle
move only for exact fixed-margin comparison, not to assert
irreducibility of a sampling chain.

\section{Canonicality of the target partition}\label{canonicality-of-the-target-partition}

\begin{theorem}[Functoriality]\label{theorem-3.1-functoriality}

If \(\varphi:G\to G'\) is an isomorphism of directed graphs, then the
induced edge bijection sends \(\Pi_{\mathrm{opp}}(G)\) to
\(\Pi_{\mathrm{opp}}(G')\). In particular, every automorphism of \(G\)
fixes \(\Pi_{\mathrm{opp}}(G)\) as an unlabelled edge partition.

\end{theorem}
\begin{proof} A directed-graph isomorphism sends directed two-step
diamond patterns to such patterns and opposite edge pairs to opposite
edge pairs. It therefore induces an isomorphism
\(\Omega(G)\cong\Omega(G')\), which maps connected components
bijectively. \end{proof}

\begin{theorem}[Universal coarsening and fixed-block-count uniqueness]\label{theorem-3.2-universal-coarsening-and-fixed-block-count-uniqueness}

Let \(\Pi^*=\Pi_{\mathrm{opp}}(G)\). Suppose that an edge partition
\(\Pi\) places every opposite pair from every directed two-step diamond
pattern in a common block. Then \(\Pi\) is a coarsening of \(\Pi^*\). If
\(\Pi\) and \(\Pi^*\) have the same number of blocks, then
\(\Pi=\Pi^*\).

\end{theorem}
\begin{proof} Every edge of \(\Omega(G)\) joins two elements of one
block of \(\Pi\). Hence every path in \(\Omega(G)\) lies in one block of
\(\Pi\), so each connected component of \(\Omega(G)\) is contained in a
block of \(\Pi\). Thus \(\Pi\) coarsens \(\Pi^*\). A strict coarsening
has fewer blocks, proving the second assertion. \end{proof}

\begin{corollary}\label{corollary-3.3}

Every nonidentity partition in the profile fibre of \(\Pi^*\) breaks at
least one generating opposite-diamond identification. It also lies
outside the \(\operatorname{Aut}(G)\)-orbit of \(\Pi^*\).
\end{corollary}

This corollary is stronger than a literal bytewise comparison of
partitions: the target orbit is a singleton because the target
construction itself is canonical.

\section{From edge blocks to relational certificates}\label{from-edge-blocks-to-relational-certificates}

\subsection{Generator relations}\label{generator-relations}

For every block \(C\in\Pi\), define a binary relation
\(T_C\subseteq V\times V\) by

\[
 (u,v)\in T_C\quad\Longleftrightarrow\quad (u,v)\in C.
\tag{4.1}
\]

We compose relations from left to right. For a nonempty word
\(w=C_1\cdots C_k\), where \(k\ge 1\), let
\(T_w=T_{C_1}\circ\cdots\circ T_{C_k}\). The finite semigroup \(S(\Pi)\)
consists of the distinct relations represented by such nonempty words.
In W2--W3 alone, we also admit the empty word \(\varepsilon\), with
\(T_\varepsilon=\operatorname{Id}_V\); it is not an additional element
of \(S(\Pi)\) unless the identity relation is represented by a nonempty
word.

\subsection{Behavioural quotient}\label{behavioural-quotient}

Start with the indiscrete partition of \(V\). At each refinement round,
two vertices remain equivalent when, for every generator \(T_C\), their
sets of reachable blocks in the current partition agree. The stable
partition is the behavioural quotient \(Q(\Pi)\). This is a standard
finite partition-refinement construction
(\citeproc{ref-PaigeTarjan1987}{Paige and Tarjan 1987};
\citeproc{ref-DorschEtAl2017}{Dorsch et al. 2017}); Appendix A gives the
precise recursion used here.

\subsection{Write-preserve-use certificates}\label{write-preserve-use-certificates}

A write-preserve-use certificate is a finite tuple of vertices,
generator names, relation words and quotient classes satisfying the
following directly checkable conditions.

\begin{enumerate}
\def\labelenumi{\arabic{enumi}.}
\tightlist
\item
  \textbf{Write.} Two transitions from one source reach vertices in
  different behavioural classes.
\item
  \textbf{Offload.} A relation word, possibly the empty word
  \(\varepsilon\), maps those two outcomes to a distinct pair \((p,q)\).
\item
  \textbf{Preserve.} Another nonidentity relation word acts as a partial
  function on the pair and gives it a nontrivial eventually periodic
  orbit that never merges the two behavioural classes.
\item
  \textbf{Writer independence after offload.} The original writer
  channels have the same enabledness profile on \(p\) and \(q\).
\item
  \textbf{Use.} A further nonidentity relation word gives \(p\) and
  \(q\) different quotient response profiles, not merely their original
  singleton profiles.
\end{enumerate}

Appendix A states the complete finite predicate. The definition does not
depend on how a witness was discovered.

Certificates with the same unordered pair of quotient classes can be
grouped into a \textbf{record system}. Its support is the union of the
pair orbits occurring in the chosen certificates.

\subsection{Translations and parity}\label{translations-and-parity}

Let \(X\) and \(Y\) be record systems, each carrying an ordered
presentation of two quotient classes. A relation word induces a
translation \(X\to Y\) when it maps all representatives of the two
source classes into the two target classes and induces a bijection
between them. Its gain is

\[
 \gamma(X\to Y)=
 \begin{cases}
 0,&\text{if the chosen order is preserved},\\
 1,&\text{if it is reversed}.
 \end{cases}
\tag{4.2}
\]

If a bit \(\eta(X)\in\mathbb Z_2\) is assigned to each record system
\(X\), then an arc \(e:X\to Y\) has transformed gain

\[
 \gamma'(e)=\gamma(e)+\eta(X)+\eta(Y).
\tag{4.3}
\]

Thus a non-loop arc flips exactly when one endpoint order is reversed,
while a loop is unchanged. The gain sum around every closed directed
walk is gauge invariant.

\subsection{The WPU-parity certificate}\label{the-wpu-parity-certificate}

We say that \((G,\Pi)\) has \textbf{a WPU-parity certificate}, written
\(\mathsf H(G,\Pi)=1\), when a finite witness supplies:

\begin{enumerate}
\def\labelenumi{\arabic{enumi}.}
\tightlist
\item
  at least two record systems built from valid write-preserve-use
  certificates;
\item
  a pair of mutually translatable record systems;
\item
  two systems whose supports are strictly nested or have nonempty proper
  overlap;
\item
  a finite reading word that distinguishes at least two phases of a
  nontrivial preservation orbit inside another record system; and
\item
  a closed directed translation walk of odd total gain.
\end{enumerate}

This is an existential certificate property. A bounded breadth-first
search is used only to discover words. A reported positive is exact once
every word and finite condition has been checked. Failure to find a
certificate is not a negative result unless the relevant relation
semigroup has been exhausted.

\subsection{Bounded discovery and certificate semantics}\label{bounded-discovery-and-certificate-semantics}

For comparison with the exploratory computation, let \(D_{B,L}^{\max}\)
denote the bounded evaluator with at most \(B\) discovered relation
products and word length at most \(L\). It constructs record systems
only from that bounded prefix. The mathematical property studied below
is the existential property \(\mathsf H\) defined above.

\begin{proposition}[One-way soundness bridge]\label{proposition-4.1-one-way-soundness-bridge}

If \(D_{B,L}^{\max}(G,\Pi)=1\), then the finite systems, words, pair
orbits and translation cycle used by that output form an
\(\mathsf H(G,\Pi)=1\) certificate once their direct predicates are
verified.

\end{proposition}
\begin{proof} Every supplied word is a product of the generator
relations. The bounded evaluator's positive output names finite WPU
witnesses, record systems, reciprocal translations, localization and
recorded-change witnesses, and an odd translation cycle. These are
precisely admissible finite objects for the existential definition;
direct recomposition checks every predicate without assuming that the
bounded relation set is complete. \end{proof}

No converse is claimed, and \(\mathsf H\) is not defined as the limit of
a particular search procedure. Bounded search is used for discovery;
every positive theorem below is supported by an explicit finite
certificate.

\section{The gain-graph reduction}\label{the-gain-graph-reduction}

For a chosen finite family of record systems, form a directed multigraph
\(\Gamma_{\mathrm{tr}}\). Its vertices are the record systems and each
verified translation is an arc labelled by its gain in \(\mathbb Z_2\).

\begin{theorem}[Parity-lift criterion]\label{theorem-5.1-parity-lift-criterion}

Define the parity lift by

\[
 \widetilde{\Gamma}_{\mathrm{tr}}
 =V(\Gamma_{\mathrm{tr}})\times\mathbb Z_2.
\]

For a record system \(X\), the following are equivalent:

\begin{enumerate}
\def\labelenumi{\arabic{enumi}.}
\tightlist
\item
  \(\Gamma_{\mathrm{tr}}\) has a closed directed walk based at \(X\)
  with odd total gain;
\item
  there is a directed path from \((X,0)\) to \((X,1)\) in the parity
  lift; and
\item
  \((X,0)\) and \((X,1)\) belong to one strongly connected component of
  the parity lift.
\end{enumerate}

Consequently, after the translations have been supplied, the parity gate
is decidable in
\(O(|V(\Gamma_{\mathrm{tr}})|+|E(\Gamma_{\mathrm{tr}})|)\) time.

\end{theorem}
\begin{proof} Replace every gain-\(g\) arc \(X\to Y\) by the two
lifted arcs \((X,p)\to(Y,p+g)\), \(p\in\mathbb Z_2\). A directed walk
lifts uniquely and its second coordinate records the sum of its gains.
It closes on the first coordinate and changes the second coordinate
precisely when its total gain is odd. Repeating the same odd closed walk
from sheet 1 returns to sheet 0, so the two lifted vertices are mutually
reachable. \end{proof}

\begin{corollary}[Potential criterion]\label{corollary-5.2-potential-criterion}

Let \(K\) be a strongly connected component of the base translation
graph. Every directed closed walk in \(K\) has even gain if and only if
there is a potential \(\varepsilon:V(K)\to\mathbb Z_2\) such that every
arc \(e:u\to v\) in \(K\) satisfies

\[
 \gamma(e)=\varepsilon(u)+\varepsilon(v).
\tag{5.1}
\]

Thus the odd gate is exactly failure of the gain labelling to be a
coboundary on some strongly connected component.

\end{corollary}
\begin{proof} Fix a root \(r\in K\). For every \(v\in K\), choose a
directed path \(P_v:r\to v\) and set \(\varepsilon(v)\) equal to its
total gain. This is independent of the chosen path: append any directed
path \(v\to r\) and compare the two resulting closed walks. Comparing
\(P_u\), an arc \(u\to v\), and \(P_v\) gives (5.1). Conversely, (5.1)
telescopes to zero around every closed walk. \end{proof}

\begin{theorem}[Transport of positive certificates]\label{theorem-5.3-transport-of-positive-certificates}

If \(\varphi:(G,\Pi)\to(G',\Pi')\) is an isomorphism of partitioned
directed graphs, allowing an arbitrary bijection of blocks, then every
witness for \(\mathsf H(G,\Pi)=1\) transports to one for \((G',\Pi')\).
Hence \(\mathsf H\) is an isomorphism invariant.

\end{theorem}
\begin{proof} Vertex relabelling conjugates every generator relation
and every word relation. It transports the behavioural quotient, forks,
pair orbits, supports, record systems and translations bijectively.
Reversing the local order at a record system changes incident gains by a
gauge transformation and does not change the parity of a closed walk.
All five existential conditions are preserved. \end{proof}

\section{Finite catalogue and computational setup}\label{finite-catalogue-and-computational-setup}

\subsection{Catalogue construction}\label{catalogue-construction}

Let terms be generated by

\[
\mathcal T ::= A\mid V\mid P(\mathcal T,\mathcal T)
                    \mid R(\mathcal T,\mathcal T),
\]

with node size \(|A|=|V|=1\) and \(|P(s,t)|=|R(s,t)|=1+|s|+|t|\). Let
\(\mathcal T_3\) contain the terms of size at most 3. An admissible
represented rewrite rule is \(R(p,q)\) with \(p,q\in\mathcal T_3\),
subject to the condition that \(V\) may occur in \(q\) only if it occurs
in \(p\). The symbol \(V\) is one consistently bound wildcard: repeated
occurrences in a pattern must match the same term. There are 79
admissible rules.

The data term is \(d\in\{A,P(A,A)\}\). A seed is the canonically sorted
multiset \(\{d,r\}\) or \(\{d,r_i,r_j\}\), where \(i\le j\) in the
lexicographic ordering of the 79 rule strings. Sorting the resulting
6478 canonical ASCII state strings lexicographically and numbering them
from zero defines the row indices.

In a current state, any represented rule \(R(p,q)\) may be applied to
any nonplaceholder occurrence matching \(p\) in any component. That
occurrence is replaced by the corresponding instance of \(q\), obtained
with the same binding of \(V\), and the components are sorted again. A
state is valid when every non-rule component is closed and no rule
contains \(V\) on its right without \(V\) on its left. Retain valid
rewrites whose changed component has size at most 255, and explore all
transitions whose source is at distance less than 3 from the seed.

For a canonical state, its \textbf{rule signature} is the ordered tuple
of components whose outermost symbol is \(R\), and its \textbf{data
signature} is the ordered tuple of all remaining components. A
qualifying recurrent core is a cyclic strongly connected component with
at least two distinct rule signatures, at least two distinct data
signatures, and an internal transition that changes the rule signature.
For each row below, the selected core is the unique largest qualifying
core. Its vertices are ordered lexicographically by canonical state
string. Its directed edges are the distinct ordered pairs of distinct
core states joined by at least one internal one-step rewrite. This
defines the loopless directed graph \(G_s\) associated with row \(s\).

\textbf{Table 6.1. Catalogue rows and selected recurrent cores.}

\begin{longtable}[]{@{}rlrr@{}}
\toprule\noalign{}
Row \(s\) & Canonical seed string \(\sigma_s\) & \(|V(G_s)|\) &
\(|E(G_s)|\) \\
\midrule\noalign{}
\endhead
\bottomrule\noalign{}
\endlastfoot
226 & \(\mathtt{\{A;R(A,R(A,A));R(V,A)\}}\) & 20 & 55 \\
3175 & \(\mathtt{\{A;R(V,A);R(V,P(A,A))\}}\) & 18 & 56 \\
3179 & \(\mathtt{\{A;R(V,A);R(V,R(A,A))\}}\) & 18 & 56 \\
3333 & \(\mathtt{\{P(A,A);R(A,P(A,A));R(P(A,V),V)\}}\) & 25 & 71 \\
3343 & \(\mathtt{\{P(A,A);R(A,P(A,A));R(P(V,A),V)\}}\) & 25 & 71 \\
3388 & \(\mathtt{\{P(A,A);R(A,P(A,A));R(V,P(A,A))\}}\) & 29 & 108 \\
4390 & \(\mathtt{\{P(A,A);R(P(A,V),V);R(V,P(A,V))\}}\) & 22 & 66 \\
4986 & \(\mathtt{\{P(A,A);R(P(V,A),V);R(V,P(V,A))\}}\) & 22 & 66 \\
6414 & \(\mathtt{\{P(A,A);R(V,A);R(V,P(A,A))\}}\) & 29 & 118 \\
6418 & \(\mathtt{\{P(A,A);R(V,A);R(V,R(A,A))\}}\) & 26 & 86 \\
\end{longtable}

The ancillary catalogue program independently reconstructs all 6478 seed
strings and the ten row-to-graph mappings in Table 6.1.

\subsection{Selection protocol}\label{selection-protocol}

The ten systems in Table 6.1 were retained by a preceding fixed screen
for multiple record systems, reciprocal translation and strict support
localization. Their target edge partitions were fixed by the
opposite-diamond rule before any profile-matched control was generated.
A preliminary bounded search with word length at most 10 found
target-positive outputs for rows \(226\) and \(3388\) within the first
1000 distinct relation products. For rows \(4390\) and \(4986\), the
cap-1000 screen had not yet found odd holonomy; both targets were
positive within the first 5000. Each of these four outputs is reproduced
by the archived deterministic target replay; the bound records the cap
at which its witness was found. The present specificity study is
conditional on those four target positives. It does not treat the ten
rows as a random population.

Rows 4390 and 4986 are mirror-isomorphic. The exact 0-based vertex
permutation is

\[
(0,1,2,3,4,5,13,14,15,20,21,18,19,6,7,8,16,17,11,12,9,10).
\tag{6.1}
\]

It maps the directed edge set and target partition of row 4390 exactly
to those of row 4986. Thus the four labelled rows represent three target
isomorphism types.

\subsection{Exact profile data}\label{exact-profile-data}

\begin{longtable}[]{@{}
  >{\raggedleft\arraybackslash}p{(\linewidth - 12\tabcolsep) * \real{0.1429}}
  >{\raggedleft\arraybackslash}p{(\linewidth - 12\tabcolsep) * \real{0.1429}}
  >{\raggedleft\arraybackslash}p{(\linewidth - 12\tabcolsep) * \real{0.1429}}
  >{\raggedleft\arraybackslash}p{(\linewidth - 12\tabcolsep) * \real{0.1429}}
  >{\raggedleft\arraybackslash}p{(\linewidth - 12\tabcolsep) * \real{0.1429}}
  >{\raggedleft\arraybackslash}p{(\linewidth - 12\tabcolsep) * \real{0.1429}}
  >{\raggedleft\arraybackslash}p{(\linewidth - 12\tabcolsep) * \real{0.1429}}@{}}
\toprule\noalign{}
\begin{minipage}[b]{\linewidth}\raggedleft
Row
\end{minipage} & \begin{minipage}[b]{\linewidth}\raggedleft
\(|V|\)
\end{minipage} & \begin{minipage}[b]{\linewidth}\raggedleft
\(|E|\)
\end{minipage} & \begin{minipage}[b]{\linewidth}\raggedleft
Diamonds
\end{minipage} & \begin{minipage}[b]{\linewidth}\raggedleft
Blocks
\end{minipage} & \begin{minipage}[b]{\linewidth}\raggedleft
Profile-fibre size
\end{minipage} & \begin{minipage}[b]{\linewidth}\raggedleft
Radius-one neighbours
\end{minipage} \\
\midrule\noalign{}
\endhead
\bottomrule\noalign{}
\endlastfoot
226 & 20 & 55 & 27 & 22 & 15,\allowbreak{}847,\allowbreak{}827,\allowbreak{}683,\allowbreak{}193,\allowbreak{}902,\allowbreak{}390,\allowbreak{}078,\allowbreak{}572,\allowbreak{}488,\allowbreak{}000,\allowbreak{}000
& 1,120 \\
3388 & 29 & 108 & 103 & 12 &
15,\allowbreak{}604,\allowbreak{}417,\allowbreak{}679,\allowbreak{}510,\allowbreak{}820,\allowbreak{}243,\allowbreak{}381,\allowbreak{}697,\allowbreak{}037,\allowbreak{}093,\allowbreak{}240,\allowbreak{}004,\allowbreak{}337,\allowbreak{}457,\allowbreak{}899,\allowbreak{}560,\allowbreak{}000 &
3,370 \\
4390 & 22 & 66 & 48 & 12 &
1,\allowbreak{}309,\allowbreak{}167,\allowbreak{}244,\allowbreak{}383,\allowbreak{}322,\allowbreak{}031,\allowbreak{}981,\allowbreak{}160,\allowbreak{}642,\allowbreak{}908,\allowbreak{}108,\allowbreak{}347,\allowbreak{}392,\allowbreak{}000 & 1,594 \\
4986 & 22 & 66 & 48 & 12 & same as 4390 & 1,594 \\
\end{longtable}

The profiles are

\[
\begin{aligned}
\lambda_{226}&=(23,3,3,3,2,2,2,2,2,1^{13}),\\
\lambda_{3388}&=(65,25,3,3,2,2,2,2,1,1,1,1),\\
\lambda_{4390}=\lambda_{4986}&=(31,7,7,6,6,2,2,1,1,1,1,1).
\end{aligned}
\tag{6.2}
\]

The fibre sizes follow from (2.1), and the neighbourhood sizes follow
from (2.5); direct enumeration of every radius-one partition gives the
same values. The automorphism groups have orders \(16,16,2,2\) for rows
\(226,3388,4390,4986\), respectively. The target orbit has size one in
every row by Theorem 3.1, while the orbits of the four displayed
radius-one countermodels have sizes \(1,4,1,1\).

\section{Radius-one countermodels}\label{radius-one-countermodels}

\subsection{Single-transposition countermodels}\label{single-transposition-countermodels}

\begin{theorem}[Local non-specificity]\label{theorem-7.1-local-non-specificity}

For each of the four labelled target-positive rows there is a partition
\(\Pi'\) such that

\[
\lambda(\Pi')=\lambda(\Pi^*),\qquad
\Pi'\sim_{\mathcal S_\lambda(E)}\Pi^*,\qquad
\mathsf H(G,\Pi')=1.
\tag{7.1}
\]

Every \(\Pi'\) is nonidentity, breaks an opposite-diamond generator, and
lies outside the automorphism orbit of \(\Pi^*\).

\end{theorem}
\begin{proof} The exact edge transpositions are listed below. For each
row, the ancillary certificate records the resulting partition and the
finite relation words witnessing \(\mathsf H=1\). Direct recomposition
verifies the edge cover, profile, transposition, co-membership distance
and odd-gain cycle. Orbit separation then follows from Corollary 3.3.
\end{proof}

\begin{longtable}[]{@{}rlrr@{}}
\toprule\noalign{}
Row & Swapped edges & \(d_{\mathrm{pair}}\) & Broken opposite pairs \\
\midrule\noalign{}
\endhead
\bottomrule\noalign{}
\endlastfoot
226 & \((0,1)\leftrightarrow(0,4)\) & 48 & 5 \\
3388 & \((0,1)\leftrightarrow(0,14)\) & 176 & 12 \\
4390 & \((0,1)\leftrightarrow(1,4)\) & 10 & 1 \\
4986 & \((0,1)\leftrightarrow(1,4)\) & 10 & 1 \\
\end{longtable}

The relation searches used to discover the displayed witnesses did not
reach a complete semigroup. This is not used as evidence of absence.
Each positive claim rests on explicit finite words and a finite gain
cycle.

\subsection{A stronger incidence-matched countermodel}\label{a-stronger-incidence-matched-countermodel}

Let \(c_s\) denote the fixed target role assignment in row \(s\), with
its actual target role labels. The stronger fibres from (2.8) exhibit an
exact rigidity--flexibility split.

\begin{theorem}[Incidence rigidity and flexibility]\label{theorem-7.2-incidence-rigidity-and-flexibility}

\[
 |\mathcal I(c_{226})|
 =|\mathcal I(c_{4390})|
 =|\mathcal I(c_{4986})|=1,
 \qquad
 |\mathcal I(c_{3388})|\ge 43.
\tag{7.2}
\]

Moreover, row 3388 has a nonidentity assignment
\(c_{\mathrm{inc}}\in\mathcal I(c_{3388})\) with
\(\mathsf H\!\left(G,\Pi(c_{\mathrm{inc}})\right)=1\). It changes 12
directed edges and is obtained by exchanging roles 0 and 1 on

\[
\begin{split}
 &(0,1),(18,1),(18,17),(14,17),(14,15),(1,15),\\
 &(1,0),(17,0),(17,18),(15,18),(15,14),(0,14).
\end{split}
\tag{7.3}
\]

\end{theorem}
\begin{proof} In the bipartite tail--head representation, the edges in
(7.3) form an alternating cycle for roles 0 and 1. Proposition 2.1
therefore proves every equality of incoming and outgoing role margins.
The explicit assignment has the same carrier and fixed role sizes as the
target and is nonidentity both with and without role labels. Direct
composition of the stored relation words verifies reciprocal
translation, strict localization, recorded change and odd gain; hence
\(\mathsf H=1\).

For rows 226, 4390 and 4986, necessary local-cardinality implications
reduce every initially possible edge-role domain to its target
singleton. The ancillary forcing traces list 9, 17 and 17 reductions,
respectively, and each step follows from the fixed local cardinalities.
Thus no non-target fixed-role assignment exists. For row 3388, the
exhaustive shortest-cycle census gives 32 distinct non-target
assignments, each preserving every fixed-role incidence margin and
changing 12 edge roles. The ten stored CSP alternatives also preserve
every margin but each changes 24 edge roles, so the two families are
disjoint. Together with the target, they certify
\(|\mathcal I(c_{3388})|\ge 1+32+10=43\), proving the lower bound in
(7.2). \end{proof}

The length 12 is minimal within the exhaustive two-role
alternating-cycle search, not among arbitrary multi-role trades. All
shortest cycles use roles 0 and 1. The search enumerates 192
starting-edge encodings; each resulting trade is counted once for each
of its six role-0 edges, yielding 32 distinct cycle trades and 32
distinct unlabelled partitions. The displayed \(c_{\mathrm{inc}}\)
belongs to this family. The other three rows have no two-role
alternating cycle, consistently with their stronger exact rigidity.
Permuting labels of equal-sized blocks does not produce a new unlabelled
partition; the singleton statement concerns any candidate whose blocks
can be aligned with the target roles so that the full rolewise margin
table agrees.

\section{A deterministic finite control family}\label{a-deterministic-finite-control-family}

For each retained row \(s\), let \(\sigma_s\) be its canonical seed
string in Table 6.1. Define a deterministic family of 64 controls as
follows. For an edge \(e=(u,v)\) and \(k\in\{0,\ldots,63\}\), rank by
the SHA-256 digest of the canonical compact UTF-8 JSON serialization of
the tuple

\[
 (\texttt{v616-profile-null},\sigma_s,k,u,v).
\tag{8.1}
\]

Break a hypothetical digest tie lexicographically by \((u,v)\). Sort the
target block sizes in nonincreasing order, cut the ranked edge list into
consecutive blocks of those sizes, sort each block, and finally
canonicalize the unlabelled block collection. This construction
preserves the exact carrier, directed edge set and complete block-size
profile. It is deterministic but is not claimed to sample the fibre
uniformly.

Write \(\Pi_k^{(s)}\) for the partition produced from row \(s\) at index
\(k\). For row 226 we abbreviate \(\Pi_k^{(226)}\) to \(\Pi_k\).

For every one of the 64 controls in each of the four rows, the bounded
evaluator used a cap of 1000 distinct relation-semigroup elements and a
maximum relation-word length of 10.

A complete calculation of all 64 declared indices in every row gives:

\begin{longtable}[]{@{}
  >{\raggedleft\arraybackslash}p{(\linewidth - 8\tabcolsep) * \real{0.2000}}
  >{\raggedleft\arraybackslash}p{(\linewidth - 8\tabcolsep) * \real{0.2000}}
  >{\raggedleft\arraybackslash}p{(\linewidth - 8\tabcolsep) * \real{0.2000}}
  >{\raggedleft\arraybackslash}p{(\linewidth - 8\tabcolsep) * \real{0.2000}}
  >{\raggedleft\arraybackslash}p{(\linewidth - 8\tabcolsep) * \real{0.2000}}@{}}
\toprule\noalign{}
\begin{minipage}[b]{\linewidth}\raggedleft
Row
\end{minipage} & \begin{minipage}[b]{\linewidth}\raggedleft
Nonidentity controls
\end{minipage} & \begin{minipage}[b]{\linewidth}\raggedleft
Bounded positive outputs
\end{minipage} & \begin{minipage}[b]{\linewidth}\raggedleft
Exact negatives
\end{minipage} & \begin{minipage}[b]{\linewidth}\raggedleft
Unresolved controls with no witness found
\end{minipage} \\
\midrule\noalign{}
\endhead
\bottomrule\noalign{}
\endlastfoot
226 & 64 & 7 & 1 & 56 \\
3388 & 64 & 19 & 0 & 45 \\
4390 & 64 & 23 & 0 & 41 \\
4986 & 64 & 29 & 0 & 35 \\
\textbf{Total} & \textbf{256} & \textbf{78} & \textbf{1} &
\textbf{177} \\
\end{longtable}

Every one of the 256 controls is a different nonidentity partition
within its row. Every positive entry is a positive output of the
declared bounded evaluator. Four radius-one positives and the target and
\(\Pi_0\) positives used in Theorem 8.1 are additionally archived as
portable word-level certificates. The 177 no-witness entries have
incomplete relation searches and are therefore left unclassified.

The observed fraction \(78/256\) is an exact fraction in this declared
finite family only. It is not an estimate of the density of positives in
the full profile fibres, whose sizes range from 35 to 53 decimal digits.

\begin{theorem}[Strict intermediate selectivity on row 226]\label{theorem-8.1-strict-intermediate-selectivity-on-row-226}

On the profile fibre of row 226,

\[
\mathsf H(G,\Pi^*)=1,
\qquad
\mathsf H(G,\Pi_0)=1,
\qquad
\mathsf H(G,\Pi_{50})=0.
\tag{8.2}
\]

Both \(\Pi_0\) and \(\Pi_{50}\) are nonidentity partitions with the same
target profile. The negative is exact: the relation semigroup of
\(\Pi_{50}\) closes after 885 elements and contains no
write-preserve-use certificate.

\end{theorem}
\begin{proof} The target and \(\Pi_0\) certificates are explicit
finite witnesses. For \(\Pi_{50}\), exhaustive closure gives 885
relations and is closed under right multiplication by every generator.
Among the 77 quotient-separated endogenous forks, allowing the identity
offload together with the 885 nonempty-word relations yields 44 distinct
offloaded pairs satisfying W2--W3. Every one of these 44 pairs admits a
nonidentity use relation, whereas none admits a nonidentity relation
whose pair orbit satisfies W4. Hence no pair supports a
write-preserve-use certificate. Since every witness for \(\mathsf H=1\)
begins with such a certificate, \(\mathsf H(G,\Pi_{50})=0\). \end{proof}

Thus \(\mathsf H\) is not determined by the profile, but it is also not
specific to the canonical target.

\section{Reproducibility}\label{reproducibility}

\subsection{Positive certificates}\label{positive-certificates}

The four radius-one results and the row-3388 incidence-matched result
are stored as explicit partitions and finite relation words. Their
verification requires only direct relation composition and the
predicates of Appendix A; it does not require repeating the bounded
discovery search. A second standard-library implementation independently
reconstructs the carriers, partitions, record systems, translations and
odd-gain cycles. The four target-positive outputs from the selection
protocol are separately recomputed by the deterministic evaluator at the
declared bounds.

\subsection{Exact finite exclusions}\label{exact-finite-exclusions}

The negative control in Theorem 8.1 includes the complete 885-element
relation semigroup and its closure data. The three incidence-rigidity
claims in Theorem 7.2 include forcing traces in which every domain
reduction follows from a local cardinality constraint. The exhaustive
alternating-cycle census certifies 32 distinct 12-edge assignments in
row 3388. The ten additional 24-edge assignments are stored explicitly,
so their disjoint contribution to the lower bound of 43 in (7.2) can be
checked without trusting a solver's summary count.

\subsection{Ancillary archive}\label{ancillary-archive}

The ancillary archive contains the explicit finite objects, independent
verification programs, exact search outputs and a SHA-256 manifest. Its
README gives standard-library commands for reproducing each claim.
Checksums verify the integrity of the archived bytes; the mathematical
evidence consists of the finite objects and the verifiable predicates
they satisfy.

\subsection{Generative-AI disclosure}\label{generative-ai-disclosure}

Text-to-text generative-AI tools were used during the preparation of
this manuscript to assist with mathematical exploration, drafting and
revising prose, translation, and code development and review. The author
takes full responsibility for the mathematical claims, references,
computations, software, and final text.

\section{Relation to previous work}\label{relation-to-previous-work}

Equivalence relations generated by opposite sides of squares occur in
domains of event structures, median graphs and relaxed square-property
relations (\citeproc{ref-Chepoi2012}{Chepoi 2012};
\citeproc{ref-HellmuthEtAl2015}{Hellmuth et al. 2015}). Our component
construction is a directed two-step specialization of that established
closure pattern. Its functoriality and universal coarsening property are
included for self-containment, not claimed as a new square-equivalence
theory.

An especially close analogy comes from true-concurrency models. In a
transition system with independence, opposite transitions in an
independence square generate an equivalence relation whose classes are
events; asynchronous transition systems instead take the event
assignment and event independence as part of the structure (see
\citeproc{ref-WinskelNielsen1995}{Winskel and Nielsen 1995, secs. 10 and 11.3.1}). Our input carries neither transition labels nor a declared
independence relation: every directed two-step diamond on four distinct
vertices present in \(G\) supplies both opposite-edge generators
automatically. Thus the resemblance is formal; we do not claim that
\(\Pi_{\mathrm{opp}}(G)\) equips \(G\) with either concurrency
structure.

Fixed-fibre perturbations are standard in conditional sampling and graph
null models (\citeproc{ref-DiaconisSturmfels1998}{Diaconis and Sturmfels
1998}; \citeproc{ref-MaslovSneppen2002}{Maslov and Sneppen 2002};
\citeproc{ref-FosdickEtAl2018}{Fosdick et al. 2018}). Markov-basis
methods formalize moves inside fibres with fixed sufficient statistics,
while degree-preserving graph switches provide a familiar network
example. At the labelled level, our fixed-profile state space is the
multislice and its swap graph is the classical transposition graph; our
unlabelled graph further quotients permutations of equal-sized role
classes (\citeproc{ref-Chase1973}{Chase 1973};
\citeproc{ref-FilmusODonnellWu2022}{Filmus, O'Donnell, and Wu 2022}).
The incidence-cycle move is likewise a standard fixed-margin switch. We
make no claim that these moves or the use of profile-matched comparison
partitions are new. The deterministic 64-index family is therefore not
described as a uniform sampler
(\citeproc{ref-ArtzyRandrupStone2005}{Artzy-Randrup and Stone 2005}).

Graph and edge-colouring equivalence is naturally taken modulo
automorphisms (\citeproc{ref-LehnerSmith2020}{Lehner and Smith 2020}).
Here the target orbit question has a short analytic answer because the
target is the component partition of a canonical auxiliary graph.
General-purpose canonical-labelling methods remain useful for
independent orbit calculations (\citeproc{ref-McKayPiperno2014}{McKay
and Piperno 2014}; \citeproc{ref-McKay1998}{McKay 1998}).

The parity gate belongs to the theory of signed and gain graphs
(\citeproc{ref-Zaslavsky1982}{Zaslavsky 1982},
\citeproc{ref-Zaslavsky1989}{1989}). An odd \(\mathbb Z_2\)-gain cycle
is the standard obstruction to balance. The contribution here is not a
new balance theorem; it is the explicit reduction of the
record-translation condition to that standard object and its use in a
certified specificity test.

The behavioural quotient is a direct relational partition refinement of
the kind treated by Paige--Tarjan and later generic coalgebraic
algorithms (\citeproc{ref-PaigeTarjan1987}{Paige and Tarjan 1987};
\citeproc{ref-DorschEtAl2017}{Dorsch et al. 2017}). The contribution is
the certificate assembled from these standard ingredients and its exact
finite countermodels, not the refinement algorithm itself.

Finally, the work follows the certifying-algorithm principle
(\citeproc{ref-McConnellEtAl2011}{McConnell et al. 2011}): each positive
output is accompanied by a finite witness that an independent program
can check.

\section{Limitations}\label{limitations}

\begin{enumerate}
\def\labelenumi{\arabic{enumi}.}
\tightlist
\item
  The four labelled target positives represent only three isomorphism
  types.
\item
  The ten-system catalogue is frozen but not a random or exhaustive
  family of directed graphs.
\item
  The complete 64-index evaluation concerns a deterministic family, not
  the full profile fibre and not a uniform sample.
\item
  Profile matching in the four local witnesses preserves only the
  multiset of block sizes. It does not, by itself, preserve per-vertex
  coloured in/out degrees or local coloured motif counts. The separate
  row-3388 witness does preserve the complete fixed-role in/out table.
\item
  Every explicit positive certificate is exact. A missing witness from
  an incomplete relation search remains unresolved.
\item
  The general theorems about component partitions, fibre counts and gain
  parity are elementary ingredients. The nontrivial content of the paper
  is their combination with the exact certified countermodels and the
  complete negative control.
\item
  The property \(\mathsf H\) is existential, not the limit of a
  particular bounded search.
\item
  The incidence-matched positive occurs in one target isomorphism type.
  The 12-edge trade is minimal only within the two-role
  alternating-cycle construction, not among arbitrary multi-role trades.
  The other three labelled rows are exactly rigid under the fixed-role
  incidence invariant.
\item
  The predicate \(\mathsf H\) and the ten-system catalogue originate in
  the exploratory construction described in Section 1; neither is
  asserted to be a standard invariant or a representative population.
  The conclusions are exact specificity statements for this fixed
  predicate and catalogue.
\end{enumerate}

\section{Conclusion}\label{conclusion}

The opposite-diamond rule defines a canonical edge partition. It is the
unique partition with the same number of blocks that preserves every
generating opposite-edge identification. Nevertheless, a finite
relational certificate built from the induced block relations does not
identify that partition in the four retained catalogue rows. The failure
occurs at the nearest possible profile-preserving scale: a single
cross-block transposition.

For row 3388 the failure also survives a substantially stronger control:
every vertex retains its exact incoming and outgoing count in every
fixed role. The companion rigidity certificates show that this stronger
fibre collapses to the target in the other three labelled rows; together
with the explicit 12-edge alternative for row 3388, they establish this
distinction exactly.

The exact negative on the same profile fibre is equally important. It
shows that the certificate is not merely a disguised function of block
sizes. Its selectivity is real but insufficient for target
identification. That is the precise conclusion supported by the
mathematics and the computations.

Future work should replace the present fixed catalogue by an
isomorph-free or parameterized family, classify the local specificity
polynomial beyond radius one, and test controls that additionally
preserve coloured motif counts.

\appendix
\section{Direct certificate predicates}\label{appendix-a.-direct-certificate-predicates}

We give the direct finite predicates used to verify the certificates.
Write \(R(v)=\{w:(v,w)\in R\}\). Relation composition is oriented so
that

\[
 (R\circ S)(v)=\bigcup_{w\in R(v)}S(w);
\]

thus \(R\) is applied first and \(S\) second.

\subsection{Behavioural refinement}\label{a.1.-behavioural-refinement}

Let \(Q_0\) be the indiscrete equivalence relation on \(V\). Given
\(Q_t\), define the signature

\[
 \sigma_t(v)=
 \bigl(\{[w]_{Q_t}:w\in T_C(v)\}\bigr)_{C\in\Pi}.
\]

Vertices are \(Q_{t+1}\)-equivalent exactly when their signatures agree.
The sequence stabilizes on a finite carrier; call the stable quotient
\(Q\), write \(\kappa_Q(v)\) for the class of \(v\), and let the split
depth of two separated vertices be their first refinement round of
separation.

\subsection{Fork and offload}\label{a.2.-fork-and-offload}

\textbf{W1 (fork).} A fork consists of a state \(s\), channels
\(A,B\in\Pi\), and states \(\ell\in T_A(s)\), \(r\in T_B(s)\) with
\(\kappa_Q(\ell)\ne\kappa_Q(r)\). The channels may coincide when one
generator relation is nondeterministic.

\textbf{W2--W3 (offload and writer independence).} An offload word
\(o\), possibly \(\varepsilon\), is admissible for the fork when \(T_o\)
is single-valued at \((\ell,r)\),

\[
 T_o(\ell)=\{p\},\qquad T_o(r)=\{q\},
\]

with \(p\ne q\), \(\kappa_Q(p)\ne\kappa_Q(q)\), and the
writer-enabledness profiles agree:

\[
 \bigl(\mathbf 1[T_C(p)\ne\varnothing]\bigr)_{C\in\{A,B\}}
 =
 \bigl(\mathbf 1[T_C(q)\ne\varnothing]\bigr)_{C\in\{A,B\}}.
\tag{A.1}
\]

Repeated channel names are removed in (A.1).

\subsection{Persistent pair orbit}\label{a.3.-persistent-pair-orbit}

\textbf{W4 (persistent orbit).} For a relation \(R\), define its partial
action on ordered state pairs by

\[
 R^{(2)}(u,v)=(u',v')
\]

when \(R(u)=\{u'\}\) and \(R(v)=\{v'\}\). Starting from \((p,q)\), a
\textbf{persistent pair orbit} for a nonempty word \(z\) with
\(T_z\ne\operatorname{Id}_V\) is an eventually periodic sequence under
\(T_z^{(2)}\) such that every iterate is defined, the two entries remain
in different \(Q\)-classes, and at least one iterate differs from its
predecessor. Finiteness of \(V^2\) bounds the verification by
\(|V|^2+1\) iterates.

\subsection{Use}\label{a.4.-use}

\textbf{W5 (use).} For a nonempty word \(u\) with
\(T_u\ne\operatorname{Id}_V\), state \(v\), and quotient \(Q\), define
the quotient response profile

\[
 \beta_u(v)=\{\kappa_Q(w):w\in T_u(v)\}.
\]

A use word for \((p,q)\) satisfies

\[
 \beta_u(p)\ne\beta_u(q),\qquad
 \bigl(\beta_u(p),\beta_u(q)\bigr)
 \ne\bigl(\{\kappa_Q(p)\},\{\kappa_Q(q)\}\bigr).
\]

A write-preserve-use certificate is a fork, an admissible offload, a
persistent pair orbit and a use word for the offloaded pair.

\subsection{Record systems}\label{a.5.-record-systems}

\textbf{W6 (support and class pair).} Every write-preserve-use
certificate has the unordered class pair

\[
 K=\{\kappa_Q(p),\kappa_Q(q)\}
\]

and an operational support equal to the union of states occurring in its
persistent pair orbit. A record system is any nonempty finite family of
certificates with one class pair \(K\). Its support is the union of
their operational supports. Its representatives in class \(k\in K\) are
all support states belonging to \(k\); both representative sets must be
nonempty.

\subsection{Translation}\label{a.6.-translation}

Let \(X,Y\) be distinct record systems and let their class pairs be
presented as ordered pairs \((x_0,x_1)\) and \((y_0,y_1)\). A nonempty
relation word \(t\) induces a translation \(X\to Y\) when, for every
\(i\in\{0,1\}\), \(T_t\) is single-valued on every representative of
\(x_i\), all images lie in the support of \(Y\), and all have one common
quotient class \(y_{\pi(i)}\), where \(\pi\) is a permutation of
\(\{0,1\}\). The gain is the parity of \(\pi\).

\subsection{The five existence gates}\label{a.7.-the-five-existence-gates}

For a finite chosen family of record systems:

\begin{enumerate}
\def\labelenumi{\arabic{enumi}.}
\tightlist
\item
  \textbf{E1, multiple systems:} at least two valid record systems are
  present.
\item
  \textbf{E2, reciprocal translation:} some unordered pair \(X,Y\) has a
  translation in both directions.
\item
  \textbf{E3, operational localization:} some two supports are nested
  strictly or have nonempty proper overlap; disjoint and equal supports
  do not qualify.
\item
  \textbf{E4, recorded change:} there exist distinct record systems
  \(X,Y\), a certificate in \(X\) whose W4 orbit has eventual period
  \(m>1\), a coordinate \(j\in\{0,1\}\), and a nonempty word \(r\).
  Write \((p_h,q_h)\), \(0\le h<m\), for the periodic part and set
  \(x_h=p_h\) if \(j=0\) and \(x_h=q_h\) if \(j=1\). For every \(h\),
  \(T_r(x_h)\) must be a singleton \(\{y_h\}\) with \(y_h\) in the
  support of \(Y\) and \(\kappa_Q(y_h)\) in its class pair; the sequence
  \((\kappa_Q(y_h))_{h=0}^{m-1}\) must be nonconstant.
\item
  \textbf{E5, odd gain:} the directed translation graph has a closed
  walk of odd total gain, equivalently the parity-lift condition of
  Theorem 5.1.
\end{enumerate}

Together these are the five conditions defining \(\mathsf H\).

\subsection{Soundness of bounded discovery}\label{a.8.-soundness-of-bounded-discovery}

The bounded program serializes the words, state pairs, persistent
orbits, supports and translation cycle used by a positive output.
Re-evaluating the relations represented by those words and checking
A.1--A.7 proves the certificate without assuming that the search
enumerated the full semigroup. Conversely, if the full finite semigroup
is enumerated and no write-preserve-use certificate exists,
\(\mathsf H\) is false because its first required object is absent.

\section{Deterministic search order}\label{appendix-b.-deterministic-search-order}

The bounded discovery algorithm performs breadth-first right
multiplication of generator relations. Relations are deduplicated
extensionally. Generator names and words are ordered lexicographically;
the discovery bounds are a maximum number of distinct relations and a
maximum word length. The algorithm is a sound positive search, not a
complete negative procedure when either cap is hit.

The SHA-index family is defined separately and is used only for the
finite evaluation in Section 8. The local transposition witnesses in
Section 7 do not depend on SHA repartitioning.

\section{Data and code map}\label{appendix-c.-data-and-code-map}

The ancillary archive contains:

\begin{itemize}
\tightlist
\item
  complete explicit certificates for the target and comparison
  partitions;
\item
  structural and orbit-verification programs;
\item
  an independent radius-one implementation of \(\mathsf H\);
\item
  the exact radius-one neighbourhood counts;
\item
  a clean-room replay of the target-positive evaluations at their
  declared per-row bounds;
\item
  the row-3388 incidence-matched positive certificate;
\item
  the exhaustive shortest alternating-cycle census for row 3388;
\item
  ten explicit row-3388 fixed-role alternatives;
\item
  the fixed-role rigidity certificates and forcing traces;
\item
  the complete 64-index evaluation output;
\item
  target and index-0 positive certificates for row 226;
\item
  the exact negative certificate for row 226, index 50;
\item
  the standard-library verification programs; and
\item
  SHA-256 checksums for every ancillary payload file.
\end{itemize}

\section*{References}\label{bibliography}
\addcontentsline{toc}{section}{References}

\protect\phantomsection\label{refs}

\end{document}